\documentclass[12pt]{amsart}
 \usepackage[dvipsnames]{xcolor}
\usepackage{mathtools}
\usepackage{amsaddr}
\usepackage{amssymb,latexsym, graphicx, enumerate, tikz, tikz-cd, caption}
 \usepackage[T1]{fontenc}

\numberwithin{equation}{section}
\newtheorem{theorem}{Theorem}[section]
\newtheorem*{theorem*}{Theorem}
\newtheorem{lemma}[theorem]{Lemma}
\newtheorem{proposition}[theorem]{Proposition}
\newtheorem{corollary}[theorem]{Corollary}

\theoremstyle{definition}
\newtheorem{definition}[theorem]{Definition}
\newtheorem*{dirichletmap}{Dirichlet's Map}

\theoremstyle{remark}
\newtheorem*{remark}{Remark}
\newtheorem*{remarks}{Remarks}
\newtheorem{example}{Example}

\DeclareMathOperator{\len}{\mathrm{length}}

\DeclareMathOperator{\SL}{\mathrm{SL}}

\DeclareMathOperator{\R}{\mathrm{R}}

\newcommand{\ca}{\textcolor{SpringGreen}}
\newcommand{\cb}{\textcolor{blue}}

\newcommand{\Kum}{\mathrm{Kum}}

\renewcommand{\theenumi}{\roman{enumi}}
\begin{document}
\title[Minimal periods in Zagier's reduction theory]{Constructing minimal periods of quadratic irrationalities in Zagier's reduction theory} 
\author{Barry~R. Smith}
\email{barsmith@lvc.edu}
\address{Department of Mathematical Sciences\\ Lebanon Valley College\\ 101 N. College Avenue, Annville, PA 17003}
\thanks{I wish to thank Keith Matthews for publishing tools for computing with binary quadratic forms at \texttt{http://numbertheory.org}.  They were an invaluable resource during this project.}

\begin{abstract}
 Dirichlet's version of Gauss's reduction theory for indefinite binary quadratic forms includes a map from Gauss-reduced forms to strings of natural numbers.  It attaches to a form the minimal period of the continued fraction of a quadratic irrationality associated with the form.  When Zagier developed his own reduction theory, parallel to Dirichlet's, he omitted an analogue of this map.  We define a new map on Zagier-reduced forms that serves as this analogue.  We also define a map from the set of Gauss-reduced forms into the set of Zagier-reduced forms that gives a near-embedding of the structure of Gauss's reduction theory into that of Zagier's.  From this perspective, Zagier-reduction becomes a refinement of Gauss-reduction. 
\end{abstract}

\maketitle

\section{Introduction}\label{S:intro}
 Dirichlet's exposition \cite{pD1854} of Gauss' reduction theory for indefinite binary quadratic forms is still the essence of the standard treatment. The core idea is to prove the main theorems by translating reduction into the regular continued fraction expansion of a quadratic irrationality.  In 1981, Zagier published an alternative reduction theory \cite{dZ1981}; his exposition is similar to Dirichlet's, replacing regular continued fractions with negative continued fractions, but it omits an important piece.   This article fills in the gap.\\[0.0cm]

\noindent \emph{Background:} A binary quadratic form is a polynomial $Ax^2 + Bxy + Cy^2$, and we assume throughout that the coefficients are rational integers.  We will typically represent a form with the shorter notation $(A,B,C)$.  Its \emph{discriminant} is $\Delta = B^2 - 4AC$, and we will only consider \emph{indefinite} forms -- those with positive, nonsquare discriminant. We call $\gcd(A,B,C)$ the \emph{content} of the form.

\begin{definition}\label{D:reduced}
An indefinite form $(A,B,C)$ is \textbf{G-reduced} if 
\begin{alignat*}{2}
	AC &< 0, & \qquad B &> |A+C|.\\
\intertext{It is \textbf{Z-reduced} if}
	A, B, C &> 0, & \qquad B &> A+C.
\end{alignat*}
\end{definition}

These are the reduced forms of Gauss and Zagier.  This definition of G-reduced forms is not the most common one, but its equivalence with Gauss's definition goes back at least to Frobenius \cite{gF1913}. We use the following notation:
\begin{align*}
	G &= \{ \, (A,B,C) \colon (A,B,C) \text{ is $G$-reduced} \, \},\\
	G^+ &= \{ \, (A,B,C) \colon (A,B,C) \in G \text{ and } A > 0 \, \},\\
	Z &= \{ \, (A,B,C) \colon (A,B,C) \text{ is $Z$-reduced} \, \}.
\end{align*}

	

We recall that every rational number $\tfrac{\alpha}{\beta} > 1$ can be expanded in two ways as a finite regular continued fraction
\begin{equation*}
	\frac{\alpha}{\beta} = q_1 + \cfrac{1}{q_2 + \cfrac{1}{\ddots \, + \cfrac{1}{q_l}}}.\\[0.1cm]
\end{equation*}
with positive integer {partial quotients} $q_1, \ldots, q_l$.  The numbers of quotients appearing in the two expansions have opposite parities: if $q_l > 2$, then the second expansion has partial quotients $(q_1, q_2, \ldots, q_l - 1, 1)$.  We recall also that if $\Delta$ is a positive, nonsquare integer, then the Pellian equation $|t^2 - \Delta u^2| = 4$ always has solutions in positive integers $(t,u)$.  The \emph{fundamental solution} is that with minimal $u$.

We now define a map that is already in Dirichlet (\cite{pD1854}, or see Section 83 of \emph{Vorlesungen \"{u}ber Zahlentheorie}): 

\begin{dirichletmap}
Let  $S_1$ be the set of nonempty \emph{natural strings} (i.e., finite sequences of natural numbers). Define
\begin{equation*}
	\gamma \colon G^+ \rightarrow S_1
\end{equation*}
as follows.  Suppose $f = (A,B,C) \in G^+$ has discriminant $\Delta$. Let $(t,u)$ be fundamental solution of the Pellian equation $|t^2 - \Delta u^2| = 4$. Set $z = \frac{t+Bu}{2}$ (an integer) and expand the rational number $\tfrac{z}{Au}$ in a finite continued fraction, choosing between the two possible expansions by requiring the parity of the number of partial quotients to be odd if $t^2 - \Delta u^2 = -4$ and even otherwise.  Define $\gamma(f)$ to be the resulting sequence of partial quotients.
\end{dirichletmap}

\begin{example}
Let $f = x^2 + 3xy - 2y^2 \in G^+$, which has discriminant 17.  The equation $t^2 - 17 u^2 = -4$ is solvable with fundamental solution $(t,u) = (8,2)$.  We compute $z = 7$ and expand $7/2$ in a continued fraction of odd length to obtain $\gamma(f) = (3,1,1)$. 
\end{example}

We define now equivalence relations on the set of strings over an alphabet and on the set of indefinite forms.  A combinatorial necklace over an alphabet $\Sigma$ is an equivalence class of finite strings over $\Sigma$ under cyclic permutation.  The name arises from the depiction, as in Figure \ref{F:stringnecklace}, of the necklace represented by the string $s_1 \cdots s_k$ as a sequence of {beads} evenly spaced around a circle.  
\begin{figure}[h]
\centering
\includegraphics[scale=0.42, trim={0cm 0cm 0cm 0}]{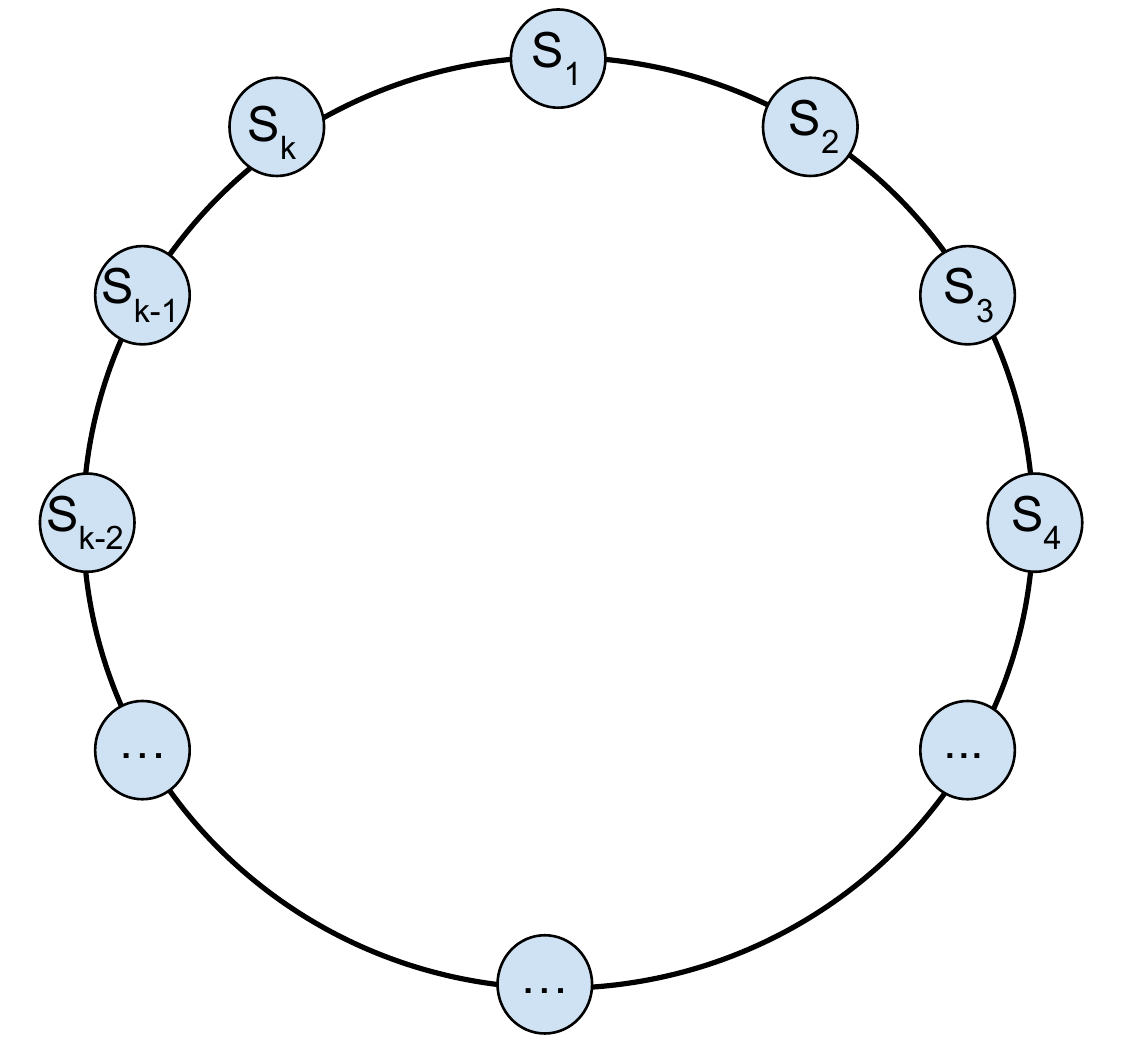}
\captionof{figure}{The necklace of the string $s_1 \cdots s_k$}
\label{F:stringnecklace}
\end{figure}
We will consider \emph{natural necklaces}, which have natural number beads, and \emph{binary necklaces}, whose beads are 0's and 1's. A nonempty string of length $n$ is \emph{primitive} if the cyclic group of order $n$ acts freely upon it.  A \emph{primitive necklace} is one whose associated strings are primitive.  Equivalently, a primitive necklace is one with no nontrivial rotational symmetry in the above depiction.

A matrix $\left[ \begin{smallmatrix} \alpha & \beta \\ \gamma & \delta \end{smallmatrix} \right]$ in $\SL_2(\mathbb{Z})$ acts on a form $f(x,y)$ as
\begin{equation*}
	f(x,y) \mapsto f(\alpha x + \beta y, \gamma x + \delta y).
\end{equation*}
Through this right-action, the set of binary quadratic forms is partitioned into equivalence classes.  All forms in a class have the same content and discriminant.  A form is \emph{primitive} if it has relatively prime coefficients.  A \emph{primitive class} is one whose forms are all primitive.

We use a subscript `$p$' on a set of forms/strings to indicate the corresponding subset of primitive forms/strings, e.g., $_{p}S_1$ denotes the set of all primitive natural strings.

We denote with $R_G$ Gauss's reduction operator on indefinite forms (see Section \ref{S:GZcompare} for the definition).  The operator $R_G$ is a permutation of the set $G$, and $R_G \circ R_G$ restricts to an permutation of $G^+$.  If $\mathfrak{c}$ is a class of forms of positive discriminant, then $\mathfrak{c} \cap G$ is nonempty and $R_G$ restricts to a permutation of $\mathfrak{c} \cap G$.   

The map $\gamma$ has many useful properties, including: 
\begin{enumerate}
	\item If $(A,B,C) \in {_{p}G^+}$ has discriminant $\Delta$, then the quadratic irrationality $\frac{B+\sqrt{\Delta}}{2A}$ has a purely periodic continued fraction and $\gamma(f)$ is its minimal period.
	  \addtocounter{enumi}{-1}
  \renewcommand\theenumi{\arabic{enumi}$^{\prime}$}
  \item $\gamma$ restricts to a bijection  $\gamma \colon {_{p}G^+} \rightarrow {_{p}S_1}$.
  \renewcommand\theenumi{\arabic{enumi}} 
	\item $R_G \circ R_G$ transports through $\gamma$ to a cyclic permutation of the corresponding natural strings.
\end{enumerate}

Property 1 was the inducement for introducing $\gamma$ to the theory.  Dirichlet works with a slightly less general map than the one given above, but a proof of the property as we have stated it is already in Weber's \emph{Lehrbuch der Algebra, Volume 1} (or see \cite[proof of Theorem 2.2.9]{fHK2013}). Property $1^{\prime}$ is an immediate consequence of Property 1.  
Property 2 shows that $\gamma$ induces a well-defined map from classes of indefinite forms  of positive discriminant to natural necklaces, and Property $1^{\prime}$ shows that this induces a map from primitive classes to primitive necklaces. \\[0.2cm]

\renewcommand{\theenumi}{\roman{enumi}}



\noindent \emph{Summary of results:} The purpose of this article is to provide an analogue of $\gamma$ for Zagier-reduced forms, a map which we call $\sigma$.  
Actually, we define two maps:
\begin{align*}
	\beta &\colon Z \rightarrow S_2\\
	\sigma &\colon Z \rightarrow B_1,
\end{align*} 
where
\begin{align*}
	S_2 &= \text{the set of natural strings of length $\geq 2$},\\ 
	B_1 &= \text{the set of binary strings with at least one 1}.
\end{align*}

\begin{definition}\label{D:beadseq}
Suppose $f = (A,B,C) \in Z$ has discriminant $\Delta$.  Let $(t,u)$ be the {fundamental solution} of the Pellian equation $|t^2 - \Delta u^2| = 4$.  Set $z = \frac{t+Bu}{2}$ (an integer) and expand the rational number $\tfrac{z}{z-Au}$ into a continued fraction, choosing between the two possible expansions by requiring the parity of the number of partial quotients to be even if $t^2 - \Delta u^2 = -4$ and odd otherwise.  Define $\beta(f)$ to be the resulting sequence of partial quotients.
We call $\beta(f)$ the \textbf{bead sequence} of $f$.
\end{definition}
The differences with Dirichlet's map $\gamma$ are that the input form $f$ is Z-reduced instead of G-reduced, the fraction expanded in a continued fraction has denominator $z-Au$ instead of $Au$, and the parity of the number of quotients is switched.  The proof of Theorem \ref{T:formfrombeads} shows, as indicated above, that the image $\beta(Z)$ is contained in $S_2$, i.e., no bead sequence has a single element. 

Before defining $\sigma$, a notational remark is in order.  We will refer to natural strings as ``strings'', but we use the more common sequence notation $(q_1, \ldots, q_l)$. We will also use binary strings and write them with the standard juxtaposition notation $q_1 \cdots q_l$.   

We define now a ``stars-and-bars'' map $\mathrm{sb} \colon S_2 \rightarrow B_1$.  Given $(q_1, \ldots q_l) \in S_2$, we arrange a sequence of $n = q_1 + \cdots + q_l$ stars in a row, then place $l-1$ bars in gaps between the stars so that the stars are divided into bunches of size $q_1$, \ldots, $q_l$ as we move left to right.  Beginning with an empty string, we traverse the stars from left to right, examining the gaps between stars and appending $0$ at each empty gap and $1$ at each bar.  The resulting binary string is $\mathrm{sb}((q_1, \ldots q_l))$.

\begin{definition}
If $f \in Z$, define $\sigma (f) = \mathrm{sb} (\beta(f))$.
\end{definition}

\begin{example}
Let $f = x^2 + 5xy + 2y^2 \in Z$, which has discriminant $17$.  The equation $t^2 - 17 u^2 = -4$ is solvable with fundamental solution $(t,u)=(8,2)$.  We compute $z=9$ and expand $9/7$ into a continued fraction of even length to obtain the bead sequence $\beta(f) = (1,3,1,1)$. The stars and bars are
\begin{equation*}
	\star \, \, \mid \, \, \star \quad \star \quad \star \, \, \mid \, \, \star \, \, \mid \,  \, \star
\end{equation*}
and the resulting binary string is
\begin{equation*}
	\sigma(f) = 10011.
\end{equation*}
\end{example}

The main results of this article are:
\begin{description}
	\item[Theorems \ref{T:rotation}, \ref{T:primitivity}, and \ref{T:Denjoy},]  which give the analogues for $\sigma$ of, respectively, Properties 2, $1^{\prime}$, and 1 of $\gamma$. 
	\item[Theorem \ref{T:xi},] which puts $\gamma$ and $\sigma$ in a commutative diagram that clarifies the link between Gauss- and Zagier- reduction. 
	\item[Theorem \ref{T:reductionrelation},] which describes how the reduction operators transport through the maps in the diagram.
\end{description}

The analogue for $\sigma$ of Property 1 states that that if $f \in {_{p}Z}$, then $\sigma(f)$ essentially gives the minimal period of the \emph{Denjoy continued fraction} \cite[pp.11-12]{aD1942}, \cite{IK2003} of a quadratic irrationality associated with $f$. Denjoy continued fractions are defined at the beginning of Section \ref{S:Denjoy}.

The analogue of Property 2 shows that Zagier-reduction translates through $\sigma$ to cyclic permutation of the corresponding binary string. Thus, $\sigma$ induces a map $\overline{\sigma}$ from classes of indefinite forms to binary necklaces, and the analogue of $1^{\prime}$ shows that this restricts to a map  from primitive classes to primitive binary necklaces.  The \emph{Z-caliber} of a class is the number of Z-reduced forms in the class.  Theorem \ref{T:primitivity} shows that $\overline{\sigma}$ is surjective, and Corollary \ref{C:Zcaliber} states that the number of beads in a primitive binary necklace is the sum of the Zagier-calibers of the primitive classes in its preimage.

Recently, Uluda\u{g}, Zeytin, and Durmu\c{s} (UZD) gave another construction of binary necklaces attached to classes of indefinite forms  \cite{UZD2016} . With each class, they associate a graph called a \emph{\c{c}ark}.  A \emph{\c{c}ark} is a quotient of a Conway topograph \cite{jC1997}, formed in such a way that the periodic river in the topograph becomes a cycle in the \c{c}ark.  We can assign an edge on the cycle to each Gauss-reduced form in the form class of the \c{c}ark.  Beginning at such an edge and traveling around the cycle, we encounter branches splitting off to the right and left.  Assigning a $1$ to branches in one direction and $0$ to branches in the other, each class then gives rise to a binary necklace.  These are not the necklaces produced by $\sigma$.

Indeed, UZD also construct a natural string by counting runs of consecutive branches that all split off in the same direction and recording a natural number for each count.   Upon translating their construction into arithmetic,  one sees that it begins the same as Dirichlet's map but expands $z/Au$ in a negative continued fraction rather than a regular one.  The rest of the UZD algorithm then amounts to performing the standard conversion from a negative continued fraction expansion to its regular continued fraction \cite{mS1864}.  Ultimately, the UZD construction is a graphical version of Dirichlet's map.

To connect the maps $\gamma$, $\beta$ and $\sigma$, we define a map $\mu \colon {G} \rightarrow Z$ by
\begin{equation}\label{E:mu}
	\mu((A,B,C)) = \begin{cases}
		(A, 2A + B, A + B + C), \qquad &\text{if $A > 0$}\\
		(A+B+C, B+2C, C), \qquad &\text{if $C > 0$.}
		\end{cases}
\end{equation}
This sends forms to equivalent forms.  Restricting $\mu$ to forms with $A > 0$ gives an injective map, as does restricting it to forms with $C > 0$, so the preimage of each form in Z has size $\leq 2$. On the other hand, there are infinitely many Z-reduced forms outside the image of $\mu$, such as $(3,7,3)$.   

Let $\eta \colon S_1 \rightarrow S_2$ be the map that prepends `1' to a string in $S_1$.  Theorem \ref{T:xi} shows that the following diagrams commute:
\begin{equation}\label{E:diagramsalt}
\begin{tikzcd}
    G^+ \arrow[hookrightarrow]{r}{\mu} \arrow{d}{\gamma} & Z \arrow{d}{\beta} \\ S_1 \arrow[hookrightarrow]{r}{\eta} & S_2
  \end{tikzcd} \qquad 
  \begin{tikzcd}
    G^+ \arrow[hookrightarrow]{r}{\mu} \arrow{d}{\gamma} & Z \arrow{d}{\sigma} \\ S_1 \arrow[hookrightarrow]{r}{\mathrm{sb} \, \circ \, \eta} & B_1
  \end{tikzcd}
\end{equation}
 
The maps $\mu$, $\eta$ and $\mathrm{sb}$ are injections.  Theorem \ref{T:reductionrelation} shows that Gauss-reduction translates through the map $\mu$ as a ``multiple of a Zagier-reduction''.  Through these diagrams, then, Zagier's theory is seen as  a refinement of Gauss's.

A final note: the present article uses and extends results developed in \cite{bS2014}.  Several conjectures were made in that article -- all of them are easy consequences of the results herein.


\section{Stringing necklaces}\label{S:construction}



In this section, we show that Zagier reduction transports through $\sigma$ to cyclic permutation of the corresponding binary string and deduce some consequences.

Zagier's reduction algorithm \cite{dZ1981} proceeds by iterating the operator that sends the form $f=(A,B,C)$ with discriminant $\Delta$ to the equivalent form $f(nx+y,-x)$, where $n$ is the ``reducing number'', defined using ceiling notation
\begin{equation}\label{E:reducingnumber}
	n = \left\lceil \frac{B+\sqrt{\Delta}}{2A} \right\rceil.
\end{equation}
We denote the Zagier-reduction operator by $R_Z$.



 
\begin{theorem*}{(Zagier, \cite[\S 13]{dZ1981})}
The orbit of an indefinite form under Zagier-reduction is eventually periodic. A form is Z-reduced if and only if its orbit is purely periodic, hence forms a cycle. There are finitely many Z-reduced forms of given discriminant, and each class of indefinite forms contains a unique cycle of Z-reduced forms.
\end{theorem*}


\begin{theorem}\label{T:rotation}
Suppose that $f$ is a Z-reduced form. The string $\sigma(R_Z \cdot f)$ is obtained from $\sigma(f)$ by cycling as follows:
\begin{enumerate}
	\item If $\sigma(f)$ begins with a `0', then the initial `0' is removed from the front and appended to the back.
	\item If $\sigma (f)$ has just one initial `1' and the succeeding characters are all `0', then the initial '1' is removed from the front and appended to the back.
	\item If $\sigma (f)$ begins with a `1' and contains at least two `1's, then the string is cycled so that the second `1' is moved to the final position.
\end{enumerate}
\end{theorem}

\begin{example}
Recall from Example 2 in Section \ref{S:intro} that $\sigma((1,5,2)) = 10011$. Iteratively applying $R_Z$ to $(1,5,2)$ gives the following cycle of forms with corresponding binary strings (obtained by applying $\sigma$):
\begin{equation*}
\begin{tikzcd}[arrows=mapsto]
		(1,5,2) \rar & (2,5,1) \rar & (4,7,2) \rar & (4,9,4) \rar & (2,7,4) \rar & (1,5,1),\\[-0.4cm]
	10011 \rar & 11001 \rar  & 00111 \rar & 01110 \rar  & 11100 \rar & 10011.
\end{tikzcd}
\end{equation*}
\end{example}

\noindent We delay proving Theorem \ref{T:rotation} until the end of this section and turn first to understanding its implications. Foremost, it shows that $\sigma$ induces a map from classes of indefinite forms to binary necklaces.   
\begin{definition}
The \textbf{length} and \textbf{weight} of a binary necklace are the number of beads and the number of $1$s in the necklace.  If $f$ is a Z-reduced form or $\mathfrak{c}$ is a class of indefinite forms, the \textbf{length} and \textbf{weight}  of $f$ or $\mathfrak{c}$ are the length and weight of the corresponding necklace.  
\end{definition}

From the definition of $\sigma$, we have:

\begin{proposition}\label{P:weightparity}
The parity of the weight of a class of indefinite forms depends only on the discriminant of the class.  A class $\mathfrak{c}$ of discriminant $\Delta$ has odd weight if $t^2 - \Delta u^2 = -4$ is solvable and even weight if it is not solvable.
\end{proposition}

When the string $\sigma(f)$ begins with a $1$, Theorem \ref{T:rotation} shows that reducing $f$ will rotate the string in such a way that it ``skips a 1''.  That is, if we iteratively cycle $\sigma(f)$ with the operation described in the theorem, the next $1$ to appear at the beginning of a string will be the third $1$ appearing in $\sigma(f)$. If ${\sigma(f)}$ has odd weight,  the strings $\sigma(R_Z^n \cdot f)$ for $n= 1$, $2$, \ldots (the exponent denoting iteration) will eventually pass through every cyclic permutation of $\sigma(f)$.  But if it has even weight, we miss some permutations.  For such discriminants, the induced map from primitive form classes to primitive necklaces is two-to-one.  To obtain a bijection, we must impose more structure on the necklaces.

\begin{definition}
An \textbf{alternating string} (resp. \textbf{alternating necklace}) is a binary string (resp. binary necklace) of even weight with $1$'s that come in two colors (we will use \ca{green} and \cb{blue}). The $1$'s alternate colors as you move along the string or necklace.  Two alternating strings represent the same alternating necklace only if they can be cycled to have $1$s and $0$'s coincide with matching colors.  The binary necklace obtained by ignoring the colors in an alternating necklace is the \textbf{underlying necklace}. An alternating necklace is \textbf{primitive} if the underlying necklace is primitive. 
\end{definition}

 We will use an overbar to denote forming equivalence classes, both of forms and of strings.   Thus, $\overline{(A,B,C)}$ and $\overline{f}$ will both denote classes of forms, while $\overline{11101}$ and $\overline{\mathfrak{s}}$ are necklaces represented by the strings  $11101$ and $\mathfrak{s}$. 
  
\begin{definition}
We denote by $\overline{\sigma}$ the map from classes of forms to necklaces defined as follows: $\overline{\sigma}(\mathfrak{c})$ is the necklace $\overline{\sigma(f)}$, where $f$ is an arbitrary Z-reduced form in $\mathfrak{c}$.   The restriction of $\overline{\sigma}$ to classes of odd weight maps to the set of binary necklaces of odd weight.  It maps classes of even weight to the set of alternating necklaces (making the arbitrary convention that when traversing the string $\overline{\sigma}(\overline{f})$ from the left, the first `1' encountered is colored $\ca{1}$).
\end{definition}

The following result links the notions of primitivity for both form classes and necklaces.
\begin{theorem}\label{T:primitivity}
If $f$ is a primitive Z-reduced form, then $\sigma(f)$ is a primitive string, and $\sigma$ restricts to a bijection
\begin{equation*}
	\sigma \colon \prescript{}{{p}}Z \leftrightarrow \prescript{}{{p}}B_1.
\end{equation*}
The induced map $\overline{\sigma}$ restricts to a bijection between the set of primitive classes of odd weight and the set of primitive binary necklaces of odd weight.  It also restricts to a bijection between the set of primitive classes of even weight and the set of primitive alternating necklaces of nonzero weight.
\end{theorem}
This will be proved immediately following the proof of Theorem \ref{T:Denjoy}.  

\begin{definition}\label{D:Zcaliber}
If $\mathfrak{c}$ is a class of indefinite forms, then the \textbf{Z-caliber} of the class is $\left| \mathfrak{c} \cap Z \right|$.
\end{definition}

From Theorems \ref{T:rotation} and \ref{T:primitivity}, we immediately obtain:
\begin{corollary}\label{C:Zcaliber}
If $\overline{\mathfrak{s}}$ is a primitive binary necklace of odd weight and length $l$, then $l = k_Z(\mathfrak{c})$, where $\mathfrak{c}$ is the unique class such that $\overline{\sigma}(\mathfrak{c}) = \overline{\mathfrak{s}}$.  If the weight is even and $\mathfrak{c}_1$ and $\mathfrak{c}_2$ are the two classes whose images under $\overline{\sigma}$ have underlying necklace $\overline{\mathfrak{s}}$, then $l = k_Z(\mathfrak{c}_1) + k_Z(\mathfrak{c}_2)$.
\end{corollary}

\begin{remark}
If define the $G$-caliber of a class $\mathfrak{c}$ to be $\left| \mathfrak{c} \cap G^+ \right|$, then the analogous result connecting $k_G(\mathfrak{c})$ with the length of the natural necklace $\overline{\gamma}(\mathfrak{c})$ is well-known to experts and follows immediately from Properties $1^{\prime}$ and 2 of $\gamma$.   
\end{remark}

The remainder of this section is devoted to the proof of Theorem \ref{T:rotation}. Let $T_Z$ be the endomorphism on the set of sequences of finite integers of length at least $2$ that operates as:
\begin{equation}\label{E:TZ}
	(q_1, q_2, \ldots, q_{l-1}, q_l) \mapsto 
	\begin{cases}
		(q_1 - 1, q_2, \ldots, q_{l-1}, q_l + 1), &\text{if $q_1 \geq 2$,}\\
		(q_2, q_1), &\text{if $q_1 = 1$ and $l=2$,}\\
		(q_3, \ldots, q_{l}, q_2, q_1), &\text{if $q_1 = 1$ and $l > 2$.}
	\end{cases}
\end{equation}

\begin{lemma}\label{L:beading}
If $f$ is a Z-reduced form, then $\len(\beta(f)) \geq 2$ and $\beta(R_Z (f)) = T_Z \circ \beta(f)$.
\end{lemma}

Before we take up the proof of Lemma \ref{L:beading}, let us note that Theorem \ref{T:rotation} is an immediate consequence of it. Recall that $\sigma = \mathrm{sb} \circ \beta$, where $\mathrm{sb}$ is the ``stars-and-bars'' map.  It is readily checked that applying $T_Z$ to $\beta(f) = (q_1, \ldots, q_l)$ transports through $\mathrm{sb}$ to the operation described in Theorem \ref{T:rotation}, the three cases arising when $q_1 \geq 2$, when $q_1 = 1$ and $l=2$, and when $q_1 = 1$ and $l > 2$ respectively.

We also make an observation that will be used frequently in what follows. 

\begin{lemma}\label{L:reducetouf}
Suppose $f$ is Z-reduced of discriminant $\Delta$.  If $(t,u)$ is the fundamental solution of $|t^2 - \Delta u^2| = 4$, then $\beta(f) = \beta(uf)$, $\sigma(f) = \sigma(uf)$, and $\gamma(f) = \gamma(uf)$.
\end{lemma}
\noindent This follows by examining the constructions of $\beta$ and $\gamma$ and observing that the fundamental solution of the Pellian equation $|x^2 - (u^2 \Delta) \cdot y^2| = 4$ is $(t,1)$.

\begin{proof}[Proof of Lemma \ref{L:beading}]

We defer the proof that $\beta(f)$ has length $\geq 2$ to the proof of Theorem \ref{T:formfrombeads}.  We can reduce the other statement to the case where $f$ has discriminant $\Delta$ of the form $k^2 \pm 4$.  Indeed, Lemma \ref{L:reducetouf} shows $\beta(f) = \beta(uf)$ with $(t,u)$ the fundamental solution of $|t^2 - \Delta u^2| = 4$.   Zagier's reduction operation commutes with scalar multiplication, so  if Lemma \ref{L:beading} holds for $uf$, then $T_Z \circ \beta(uf) = \beta(u R_Z(f))$.  Then $T_Z \circ \beta(f) = \beta (R_Z(f))$ and the reduction is complete. The lemma might now be proved directly using properties of continued fractions, but we will instead use a result from an earlier work \cite{bS2014}. 

Suppose the Z-reduced form $f=Ax^2 + Bxy + Cy^2$ has discriminant $\Delta = k^2 +(-1)^s \cdot 4$, with $s = 0$ or $1$.  The Pellian equation $|t^2 - \Delta u^2| = 4$ then has fundamental solution $(t,u) = (k,1)$, so to compute $\beta(f)$ we set $z = \frac{k+B}{2}$. We recall the definition of $\beta(f)$ and also define a modification $\tilde{\beta}(f)$:
\begin{itemize}
	\item  $\beta(f)$ is the sequence of length parity $s$ obtained by expanding $\frac{z}{z-A}$ as a simple continued fraction.
	\item  $\tilde{\beta}(f)$ is the sequence of length parity $s$ obtained by expanding $\frac{z}{A}$ as a simple continued fraction.
\end{itemize}
(We say a sequence has length parity $s$ if its length has the same parity as $s$.)

Let us say we \emph{pinch} the left end of a finite sequence of positive integers by transforming it through the rule
\begin{equation*}
	(q_1, q_2, q_3, \ldots q_l) \mapsto \begin{cases}
		(1, q_1-1, q_2, q_3, \ldots, q_l),  &\text{if $q_1 \geq 2$,}\\
		(q_2+1,q_3,\ldots, q_l) &\text{if $q_1=1$.}
		\end{cases}
\end{equation*}
We pinch the right end similarly. We also make the convention that the sequence `$(1)$' and the empty sequence are pinched by doing nothing.

We \emph{knead} a finite sequence of positive integers by
\begin{enumerate}
	\item removing the leftmost entry, then
	\item pinching both ends of what remains, then
	\item placing the removed entry on the right end of the result.
\end{enumerate}
Theorem 2 of \cite{bS2014} shows that Zagier-reduction transports through the map $\tilde{\beta}$ to kneading the corresponding sequence.

We claim that pinching both ends of $\tilde{\beta}(f)$ produces $\beta(f)$. Comparing the first few steps of the Euclidean algorithm with $z$ and $A$ and then with $z$ and $z - A$ quickly reveals that one sequence of quotients is obtained from the other by pinching the left end.  Pinching just the left end switches the length parity of the sequence of quotients, so to compute ${\beta}(f)$ from $\tilde{\beta}(f)$, we must change the parity again by switching to the other of the two simple continued fraction expansions of $\tfrac{z}{z-A}$.  Since switching between the two continued fraction expansions is accomplished by pinching the right end of the sequence of quotients, the claim follows. 

 To conclude, we must see that acting with $T_Z$ on a sequence $(q_1, \ldots, q_l)$ with $l \geq 2$ has the same effect as performing the three steps
\begin{enumerate}
	\item pinch both ends, 
	\item knead the result, then 
	\item pinch both ends again.
\end{enumerate}
 There are several cases to consider.  First, we assume $q_1 \geq 2$ and $q_l = 1$. If $l \geq 3$, then depending on whether $q_1 = 2$ or $q_1 \geq 3$, the above three steps look as follows:
 \begin{align*}
 	&(2, q_2, \ldots, q_{l-1}, 1) \mapsto (1, 1, q_2, \ldots, q_{l-1} + 1) \mapsto
	(q_2 + 1, \ldots, q_{l-1}, 1, 1) \mapsto (1, q_2, \ldots, q_{l-1}, 2)\\
	&(q_1, \ldots, q_{l-1}, 1) \mapsto (1, q_1 - 1, \ldots, q_{l-1} + 1) \mapsto
	(1, q_1 - 2, \ldots, q_{l-1}, 1, 1) \mapsto (q_1 - 1, \ldots, q_{l-1}, 2)
\end{align*}
 If $l=2$, then depending on whether $q_1 = 2$ or $q_1 \geq 3$, the steps are instead
 \begin{align*}
 	&(2, 1) \mapsto (1, 2) \mapsto (2,1) \mapsto (1,2)\\
	&(q_1, 1) \mapsto (1,q_1) \mapsto (1,q_1 - 2, 1, 1) \mapsto (q_1 - 1, 2)
\end{align*}
If instead $q_1 \geq 2$ and $q_l \geq 2$, we have in the cases $q_1 = 2$ or $q_1 \geq 3$
\begin{align*}
	&(2, q_2, \ldots, q_l) \mapsto (1, 1, q_2, \ldots, q_{l} - 1, 1) \mapsto (q_2 + 1, \ldots, q_l, 1) \mapsto (1, q_2, \ldots, q_l + 1),\\
	&(q_1, \ldots, q_l) \mapsto (1, q_1 - 1, \ldots, q_l - 1, 1) \mapsto (1, q_1 - 2, \ldots, q_l, 1) \mapsto (q_1 - 1, q_2, \ldots, q_{l-1}, q_l + 1)
\end{align*} 
It remains to consider when $q_1 = 1$.  First let us consider when $l \geq 4$.  If $q_3 = 1$ and $q_l = 1$, then we have, depending on whether $l=4$ or $l \geq 5$, either of the following:
\begin{align*}
	(1, q_2, 1, 1) &\mapsto (q_2 + 1, 2) \mapsto (2, q_2 + 1) \mapsto (1,1,q_2,1)\\
	(1, q_2, 1, q_4 \ldots, q_{l-1}, 1) &\mapsto (q_2 + 1, 1, q_4 \ldots, q_{l-1} + 1) \mapsto
	(q_4 + 1 \ldots, q_{l-1}, 1, q_2 + 1)\\
	& \mapsto (1, q_4, \ldots, q_{l-1}, 1, q_2, 1)
\end{align*}
If instead $q_3 \geq 2$ and $q_l = 1$, we have
\begin{align*}
	(1, q_2, q_3, \ldots, q_{l-1}, 1) &\mapsto (q_2 + 1, q_3, \ldots, q_{l-1} + 1) \mapsto
	(1, q_3 - 1, \ldots, q_{l-1}, 1, q_2 + 1)\\
	 &\mapsto (q_3, \ldots, q_{l-1}, 1, q_2, 1)
\end{align*}
If $q_3 = 1$ and $q_l \geq 2$, we have
\begin{align*}
	(1, q_2, 1, q_4, \ldots, q_l) &\mapsto (q_2 + 1, 1, q_4, \ldots, q_{l}-1, 1) \mapsto (q_4 + 1, \ldots, q_l, q_2 + 1)\\
	&\mapsto (1, q_4, \ldots, q_l, q_2, 1)
\end{align*}
Finally, if $q_3 \geq 2$ and $q_l \geq 2$, we have
\begin{align*}
	(1, q_2, q_3, \ldots, q_l) &\mapsto (q_2 + 1, q_3, \ldots, q_l - 1, 1) \mapsto (1, q_3 - 1, \ldots, q_l, q_2 + 1)\\ 
	&\mapsto (q_3, \ldots, q_l, q_2, 1)
\end{align*}

When $l=3$, we have, when $q_3=1$ or $q_3 \geq 2$ respectively
\begin{align*}
	(1, q_2, 1) &\mapsto (q_2+2) \mapsto (q_2+2) \mapsto (1, q_1, 1)\\
	(1, q_2, q_3) &\mapsto (q_2 + 1, q_3 - 1, 1) \mapsto (1, q_3 - 1, q_2 + 1) \mapsto (q_3, q_2, 1)
\end{align*}
And finally, when $l=2$, we have
\begin{align*}
	(1, q_2) &\mapsto (q_2, 1) \mapsto (1,q_2) \mapsto (q_2, 1)
\end{align*}
In every case, the net result is the operation $T_Z$.
\end{proof}

\section{Sections of $\beta$ and $\sigma$}\label{S:continuants}

Let us consider what it takes to invert the string map ${\sigma}$.   The stars-and-bars map $\mathrm{sb}$ is invertible, so from a given binary string $\mathfrak{s}$, we can produce a unique bead sequence. The map $\beta$ from Z-reduced forms to bead sequences is generally many-to-one, but we can find a section using continuants.   We introduce these now, as well as the fundamental identities that will be used heavily through the rest of the paper.

Let $\left[ q_1, \ldots, q_l \right]$ and $\left\langle q_1, \ldots, q_l \right\rangle$ denote, respectively, the numerator and denominator when the regular continued fraction with quotients $\left(q_1, \ldots, q_l \right)$ is simplified to a reduced fraction.  Then
\begin{equation}\label{E:definingcontinuants}
	\frac{\left[ q_1, \ldots, q_l \right]}{\left\langle q_1, \ldots, q_l \right\rangle} = q_1 + \frac{1}{\frac{\left[ q_2, \ldots, q_l \right]}{\left\langle q_2, \ldots, q_l \right\rangle}} = \frac{q_1 \left[ q_2, \ldots, q_l \right] + \left \langle q_2, \ldots, q_l \right\rangle}{\left[ q_2, \ldots, q_l \right]}
\end{equation}
Since $\left[q_2, \ldots, q_l \right]$ and $\left< q_2, \ldots, q_l \right>$ are relatively prime, the fraction on the right is reduced.  Generally, then, $\left< q_1, \ldots, q_l \right> = \left[ q_2, \ldots, q_l \right]$ and $\left[ q_1, \ldots, q_l \right]$ and $\left[ q_2, \ldots, q_l \right]$ are coprime.

\begin{definition}
 A \textbf{continuant} is a number attached to a natural string, denoted by
\begin{equation*}
	\left[ q_1, \ldots, q_l \right]
\end{equation*}
and computed as the numerator when the regular continued fraction with partial quotients $q_1$ \ldots, $q_l$ is simplified to a reduced fraction.  From \eqref{E:definingcontinuants}, the corresponding denominator is then $\left[ q_2, \ldots, q_l \right]$.  We will encounter identities that specialize to involve continuants of the form $[q_s, q_{s-1}]$ and $[q_s, q_{s-2}]$.  We make the convention that the former equals $1$ and the latter equals $0$.
\end{definition}

Equation \eqref{E:definingcontinuants} gives the recursion $\left[ q_1, \ldots, q_l \right] = q_1 \left[ q_2, \ldots, q_l \right] + \left[ q_3, \ldots, q_l \right]$ for $l \geq 1$.  This recursion shows that $q_1$ is the quotient and $\left[ q_3, \ldots, q_l \right]$ is the remainder when dividing $\left[ q_1, \ldots, q_l \right]$ by $\left[ q_2, \ldots, q_l \right]$.  Consequently, if $\alpha = \left[ q_1, \ldots, q_l \right]$ and $\beta = \left[ q_2, \ldots, q_l \right]$, then $\alpha/\beta$ expands in a continued fraction with quotient sequence $(q_1, \ldots, q_l)$.  

From the above recursion, we find by induction the matrix identity
\begin{equation}\label{E:matrixidentity}
	\begin{bmatrix} q_1 & 1\\ 1 & 0 \end{bmatrix} \begin{bmatrix} q_2 & 1\\ 1 & 0 \end{bmatrix} \cdots \begin{bmatrix} q_l & 1\\ 1 & 0 \end{bmatrix} = \begin{bmatrix} \left[ q_1, \ldots, q_l \right] & \left[ q_1, \ldots, q_{l-1} \right]\\ \left[ q_2, \ldots, q_{l} \right] & \left[ q_2, \ldots, q_{l-1} \right] \end{bmatrix}
\end{equation}
Transposing, we find
\begin{equation}\label{E:contsymmetry}
	\left[ q_l, \ldots, q_1 \right] = \left[ q_1, \ldots, q_l \right],
\end{equation}
Taking determinants instead, we have
\begin{equation}\label{E:contdet}
	\left[ q_1, \ldots, q_l \right] \left[ q_2, \ldots, q_{l-1} \right] - \left[ q_1, \ldots, q_{l-1} \right] \left[ q_2, \ldots, q_l \right] = (-1)^l.
\end{equation} 


From \eqref{E:contsymmetry}, we find there are two recursions:
\begin{equation}
\begin{aligned}\label{E:contrecursions}
	\left[ q_1, \ldots, q_l \right] &= q_1 \left[ q_2, \ldots, q_l \right] + \left[ q_3, \ldots, q_l \right],\\
	\left[ q_1, \ldots, q_l \right] &= q_l \left[ q_1, \ldots, q_{l-1} \right] + \left[ q_1, \ldots, q_{l-2} \right],
\end{aligned}
\end{equation}
which have the useful consequences
\begin{equation}\label{E:contends}
\begin{aligned}
	\left[ q_1+q, q_2, \ldots, q_l \right] &= \left[ q_1, \ldots, q_l \right] + q\left[ q_2, \ldots, q_l \right],\\ 
	\left[ q_1, \ldots, q_{l-1}, q_l + q \right] &= \left[ q_1, \ldots, q_l  \right] + q \left[ q_1, \ldots, q_{l-1} \right],
\end{aligned}
\end{equation}
and
\begin{equation}\label{E:contones}
\begin{aligned}
	\left[ 1, q_2, \ldots, q_l \right] &= \left[ q_2+1, \ldots, q_l \right], \\
	\left[ q_1, \ldots, q_{l-1}, 1 \right] &= \left[ q_1, \ldots, q_{l-1} + 1 \right].
\end{aligned}
\end{equation}

We will find it useful to define a continuant with $0$ as an entry.  We set $\left[0 \right] = 0$ and
\begin{equation}\label{E:0cont}
	\left[ 0, q_2, \ldots, q_l \right] = \left[ q_3, \ldots, q_l \right], \qquad \left[ q_1, \ldots, q_{l-1}, 0 \right] = \left[ q_1, \ldots, q_{l-2} \right]
\end{equation}
for $l \geq 2$.  Identities \eqref{E:matrixidentity} through \eqref{E:contones} then remain true with continuants having $0$ as an end entry.

We now return to considering a section of $\beta$.   Given a finite sequence of positive integers $(q_1, \ldots, q_l)$ with $l \geq 2$, we produce a form $(A,B,C)$ by setting
\begin{align*}
	A &= \left[ q_1-1, q_2, \ldots, q_l \right]\\
	C &= \left[ q_1, \ldots, q_{l-1}, q_l - 1 \right]\\
	B &= \left[ q_1, \ldots, q_l \right] + \left[ q_{1} -1, q_2, \ldots, q_{l-1}, q_l - 1 \right],
\end{align*}
(A continuant with end entry equal to $0$ is defined as in \eqref{E:0cont}.) Let us denote the form with these coefficients by
\begin{equation}\label{E:tau}
	\tau((q_1, \ldots, q_l)).
\end{equation}
Recall that $S_2$ is the set of natural strings with length $\geq 2$.  Let $\tilde{Z}$ be the set of Z-reduced forms with discriminant of the form $k^2 \pm 4$.

\begin{theorem}\label{T:formfrombeads}
The map $\tau \colon S_2 \rightarrow \tilde{Z}$ is a bijection, with inverse given by the restriction of $\beta$ to $\tilde{Z}$. In fact, the discriminant of $\tau((q_1, \ldots, q_l))$ is
\begin{equation*}
	\left( \left[ q_1, \ldots, q_l \right] - \left[ q_1 - 1, q_2, \ldots, q_{l-1}, q_l - 1 \right] \right)^2 + (-1)^l \cdot 4
\end{equation*}
\end{theorem}
 
\begin{corollary}\label{C:sigmasurjects}
The map $\sigma \colon Z \rightarrow B_1$ is a surjection.  The map $\tau \circ \mathrm{sb}^{-1}$ is a section of $\sigma$ whose image is the set of forms in $Z$ with discriminants of the form $k^2 \pm 4$.
\end{corollary}

\begin{proof}
The following identity follows by using $[1, q_1-1, \ldots, q_l-1, 1]$ in place of $[q_1, \ldots, q_l]$ in \eqref{E:contdet} and then applying \eqref{E:contones}:
\begin{equation}\label{E:contdetalt}
	\left[ q_1, \ldots, q_l \right] \left[ q_1 - 1, q_2, \ldots, q_{l-1}, q_l - 1 \right]  - \left[ q_1, \ldots, q_{l-1}, q_{l} - 1 \right]  \left[ q_1 - 1, q_2, \ldots, q_{l} \right] = (-1)^l.
\end{equation}
Using it, we compute the discriminant of $\tau((q_1, \ldots, q_l))$ to be
\begin{align*}
		([q_1, \ldots, q_l] &+ [q_1-1, q_2, \ldots, q_{l-1}, q_l - 1] )^2 - 4 \left[ q_1 - 1, q_2, \ldots, q_l \right] \left[ q_1, \ldots, q_{l-1}, q_l - 1 \right]\\
		&=([q_1, \ldots, q_l] - [q_1-1, q_2, \ldots, q_{l-1}, q_l - 1] )^2 + 4 \cdot (-1)^l.
\end{align*}		

Now suppose we have a sequence $(q_1, \ldots, q_l)$ with $l \geq 2$,  and set $\tau((q_1, \ldots, q_l)) = (A,B,C)$. The above discriminant is positive.  This is immediate if $l$ is even.  Otherwise, $l \geq 3$, and \eqref{E:contends} gives
\begin{align*}
	[ q_1, \ldots, q_l ] - \left[ q_1 - 1, q_2 \ldots, q_{l-1}, q_l - 1 \right] &> \left[ q_1, \ldots, q_l \right] - \left[ q_1, \ldots, q_{l-1}, q_l - 1 \right]\\
	&= \left[ q_1, \ldots, q_{l-1} \right] \geq 2.
\end{align*}
To check that $(A,B,C)$ is Z-reduced, we note that $A$, $C > 0$.  Using \eqref{E:contends} again,
\begin{align*}
	&B - A - C\\
	 &= [ q_1, \ldots, q_l] + \left[ q_1 - 1, q_2, \ldots, q_{l-1}, q_l - 1 \right] - \left[ q_1 - 1, q_2, \ldots, q_l \right] - \left[ q_1, \ldots, q_{l-1}, q_l - 1 \right]\\
	&= [ q_1, \ldots, q_{l-1} ] - \left[ q_1 - 1, q_2, \ldots, q_{l-1} \right] = \left[ q_2, \ldots, q_{l-1} \right] > 0.
\end{align*}
We may thus apply $\beta$ to $\tau((q_1, \ldots, q_l))$.  

Using the discriminant of $\tau((q_1, \ldots, q_l))$ found above, $\beta \circ \tau ((q_1, \ldots, q_l))$ is computed by expanding 
\begin{equation*}
	\frac{z}{z-A} =  \frac{[q_1, \ldots, q_l]}{[q_1, \ldots, q_l] - \left[ q_1 - 1, q_2, \ldots, q_l \right]} = \frac{[q_1, \ldots, q_l]}{[q_2, \ldots, q_l]}
\end{equation*}
in a continued fraction of length parity $l$.  Thus, $\beta \circ \tau ((q_1, \ldots, q_l)) = (q_1, \ldots, q_l)$.

Suppose now that $(A,B,C)$ is Z-reduced of discriminant $\Delta = k^2 + (-1)^s \cdot 4$ with $k > 0$ (using $k=1$ and $s=0$ when $(A,B,C) = (1,3,1)$). The sequence $\beta((A,B,C))$ is found by setting $z = \frac{B+k}{2}$ and expanding $\frac{z}{z-A}$ as a continued fraction with quotients $(q_1, \ldots, q_l)$ chosen so that $l$ and $s$ have the same parity.  {We suppose, for now, that $l \geq 2$.  We will see at the end of the proof that $l = 1$ cannot happen.} 

We have 
\begin{equation}\label{E:zcong}
	\frac{B^2 - k^2}{4} = AC + \frac{\Delta - k^2}{4} = AC + (-1)^l,
\end{equation}
so $z-A$ and $z$ are coprime.  It follows that $z = [q_1, \ldots, q_l]$ and $z-A = [q_2, \ldots, q_l]$.  Thus, $A = \left[q_1 - 1, q_2, \ldots, q_l \right]$, and in particular, $0 < A < z$. 

We also have
\begin{equation}\label{E:Zagierinequality}
	k^2 - (B-2A)^2 = \Delta - (B-2A)^2 - (-1)^l \cdot 4 = 4A(B-A-C) - (-1)^l \cdot 4 \geq 0.
\end{equation}
If the inequality is strict, then $A > \frac{B-k}{2}$, and so from \eqref{E:zcong} we have $Az > AC + (-1)^l$ and then $z \geq C$.  Were it true that $z = C$, we would have $k = 2C - B$ and
\begin{equation*}
	k^2 - (2C - B)^2 = 4C(B-A-C) - (-1)^l \cdot 4 = 0.
\end{equation*}
Then $C=1$ and $B+k = 2$, but this cannot happen since $B \geq 3$ in a Z-reduced form.  Thus, $C < z$.  Otherwise, when \eqref{E:Zagierinequality} is an equality, then $A = 1$, $B = C+2$, and $l$ is even, so \eqref{E:zcong} shows  $k = C$.  Thus,  $z = C + 1$, and so $0 < C < z$.  

We obtain from \eqref{E:contdetalt} and \eqref{E:zcong} the congruences 
\begin{equation*}
	A \left[ q_1, \ldots, q_{l-1}, q_l - 1 \right] \equiv AC \equiv (-1)^{l+1} \pmod{z}. 
\end{equation*} 
Since $0 < C < z$, we conclude $C = \left[ q_1, \ldots, q_{l-1}, q_l - 1 \right]$. 

Our expressions for $A$ and $C$ now show that $\tau \circ \beta ((A,B,C))$ has the form $(A,B',C)$ with some middle coefficient $B'$.  From \eqref{E:contdetalt} and \eqref{E:zcong}, we also find that $\frac{B-k}{2} = \left[ q_1 - 1, q_2, \ldots, q_{l-1}, q_l - 1 \right]$, and hence
\begin{equation*}
	k = z - \frac{B-k}{2} = \left[ q_1, \ldots, q_l \right] - \left[ q_1 - 1, q_2, \ldots, q_{l-1}, q_l - 1 \right].
\end{equation*}
Our determination of the discriminant of $\tau((q_1, \ldots, q_l))$ then gives
\begin{equation*}
	B'^2 - 4AC = k^2 + 4 \cdot (-1)^l = B^2 - 4AC,
\end{equation*}
so $\tau \circ \beta ((A,B,C)) = (A,B,C)$.

Reworking the above argument assuming $l = 1$, we find $z = q_1$, $A = q_1 - 1$, then $C = q_1 - 1$, and finally $\tfrac{B-k}{2} = q_1 - 2$.  Then $B = 2q_1 - 2$, contradicting $B > A+C$.  It follows that $\beta \colon \tilde{Z} \rightarrow S_2$ is a bijection and $\tau$ is its inverse.
\end{proof}

\section{Comparing Gauss- and Zagier-reduction}\label{S:GZcompare}

This section clarifies the simple relationship between the beading map $\beta$ and Dirichlet's map $\gamma$ and justifies the statement that Zagier-reduction refines Gauss-reduction.  

Our presentation of Gauss-reduction is a modification of the usual algorithm, moving through the cycle of reduced forms in the reverse direction.  This puts the algorithm in a form parallel to Zagier's algorithm, which is necessary for the statement of Theorem \ref{T:reductionrelation}.  We will only apply Gauss's reduction algorithm to G-reduced forms.  Recall that these forms were defined in Definition \ref{D:reduced}.

So, suppose $f = (A,B,C)$ is G-reduced.  Compute $\delta$ such that 
\begin{equation*}
	|\delta| = \left\lfloor \frac{B+\sqrt{\Delta}}{2|A|} \right\rfloor
\end{equation*}
with the sign of $\delta$ chosen so that $\delta A > 0$.  Our version of Gauss's reduction operator sends $f$ to the form $f\left( \delta x + y, -x \right)$.  This switches the sign of $A$.

Recall the definitions of $S_1$, $S_2$, $B_1$, $G$, $G^+$, and $Z$ from Section \ref{S:intro}, and that $G^+$ is the domain of $\gamma$. We also define
\begin{equation*}
	G^- = \{ \, (A,B,C) \mid (A,B,C) \in G, A < 0 \, \}.
\end{equation*}


We also define maps
\begin{itemize}
	\item $\eta^+ \colon S_1 \rightarrow S_2$, which prepends a $1$ to $\mathfrak{s} \in S_1$,
	\item   $\eta^- \colon S_1 \rightarrow S_2$, which appends a $1$ to $\mathfrak{s} \in S_1$,
	\item $\R \colon S_1 \rightarrow S_1$, defined so that $(q_1, \ldots, q_l)^{\R} = (q_l, \ldots, q_1)$,
	\item $T_G \colon S_1 \rightarrow S_1$, defined so that $T_G ((q_1, \ldots, q_l)) = (q_2, \ldots, q_l, q_1)$,
	\item $T_Z \colon S_2 \rightarrow S_2$ defined by \eqref{E:TZ}.
	\item  $\mu \colon G \rightarrow Z$ defined by \eqref{E:mu},
	\item $\rho \colon G \rightarrow G$, defined so that $\rho ((A,B,C)) = (-A,B,-C)$,
	\item $R_G \colon G \rightarrow G$, the Gauss-reduction operator,
	\item $R_Z \colon Z \rightarrow Z$, the Zagier-reduction operator.
\end{itemize}

We also define an operator $\R$ on the set of all binary quadratic forms by $(A,B,C)^{\mathrm{R}} = (C,B,A)$.  The reader should note from context whether $\R$ is being applied to a string or to a form. Please note as well
\begin{itemize}
	\item $\R$, $\rho$, and $R_G$ induce bijections $G^+ \rightarrow G^-$ and $G^- \rightarrow G^+$ (which we denote also with $\R$, $\rho$, and $R_G$),
	\item $f$ and $\mu(f)$ are equivalent forms,
	\item $f$ and $\rho(f)$ are equivalent if $f$ is in a class of odd weight, but otherwise are in different classes,
	\item $f$ and $f^{\R}$ are often not equivalent (in fact, they are in inverse classes in the class group).
\end{itemize}

\begin{proposition}\label{P:reversal}
\leavevmode
\begin{enumerate}
	\item If $f \in G^-$, then $\gamma(f^R) = \left( \gamma \circ \rho(f) \right)^{\mathrm{R}}$.
	\item If $f \in Z$, then $\beta \left(f^{\mathrm{R}} \right) = \beta(f)^{\mathrm{R}}$.
\end{enumerate}
\end{proposition}

\begin{proof}
Let $f = (A,B,C) \in G^-$ have discriminant $\Delta$, and let $(t,u)$ be the fundamental solution of $|t^2 - \Delta u^2| = 4$.  If $z = \frac{t+Bu}{2}$, then $\gamma(f^{\R}) = (q_1, \ldots, q_l)$ is the sequence of partial quotients when expanding $z/Cu$ in a continued fraction. Since
\begin{equation}\label{E:AuCucongruence}
	\frac{t-Bu}{2} z = \frac{t^2 - B^2u^2}{4} = (-1)^l - ACu^2,
\end{equation}
we have $Au \cdot Cu \equiv (-1)^{l} \pmod{z}$ and $\gcd(z, Cu) = 1$, so also $z=\left[ q_1, \ldots, q_l \right]$ and $Cu = \left[ q_2, \ldots, q_l \right]$.
From \eqref{E:contdet}, we then have $Cu \left[q_1, \ldots, q_{l-1} \right] = (-1)^{l+1} \pmod{z}$, and it follows that 
\begin{equation}\label{E:Aucongruence}
	-Au\equiv \left[ q_1, \ldots, q_{l-1} \right] \pmod{z}.
\end{equation}

We show now that $-Au$ and $\left[ q_1, \ldots, q_{l-1} \right]$ are both between $1$ and $z$.  This is clear for the latter, and $-Au \geq 1$ since $f \in G^-$.  Let us suppose that $-Au > z$.     Then \eqref{E:AuCucongruence} implies $\frac{t-Bu}{2} z - Cu z \geq 0$, so $t \geq (B+2C)u$.  Squaring both sides and simplifying, we deduce
\begin{equation*}
	t^2 - \Delta u^2 \geq 4 C u^2 (B+A+C).
\end{equation*}
Since $f \in G^-$, we know the right side is positive.  Then $t^2 - \Delta u^2 = 4$ and we have equality, so $u=1$, $C=1$, and $B+A+C = 1$, so $B = -A$ and $t= 2C-A$.    Then $z = C-A$, contradicting the assumption that $-Au > z$.

We now know that $-Au$ and $\left[ q_1, \ldots, q_{l-1} \right]$ are both between $1$ and $z$.  In light of \eqref{E:Aucongruence}, they are equal.  From \eqref{E:contsymmetry}, $-Au = \left[ q_{l-1}, \ldots, 1 \right]$ and $z = \left[ q_l, \ldots q_1 \right]$.  To compute $\gamma \circ \rho (f) $, we expand $z/(-Au)$ in a continued fraction with quotient sequence of length parity $l$.  Thus, $\gamma \circ \rho (f)= (q_l, \ldots, q_1)$, and the proof of (i) is complete.  We omit the details of the very similar argument for (ii).
\end{proof}

The next theorem gives the precise relationship between $\gamma$ and the beading map $\beta$ and is the key to connecting results about Z-reduced forms with results about G-reduced forms.

\begin{theorem}\label{T:xi}
The following diagrams commute
\begin{equation}\label{E:diagrams}
\begin{tikzcd}
    G^+ \arrow[hookrightarrow]{r}{\mu} \arrow{d}{\gamma} & Z \arrow{d}{\beta} & \\ S_1 \arrow[hookrightarrow]{r}{\eta^+} & S_2 \arrow[hookrightarrow]{r}{\mathrm{sb}} & B_1
  \end{tikzcd} \qquad 
  \begin{tikzcd}
    G^- \arrow[hookrightarrow]{r}{\mu} \arrow{d}{\gamma \circ \rho} & Z \arrow{d}{\beta} & \\ S_1 \arrow[hookrightarrow]{r}{\eta^-} & S_2 \arrow[hookrightarrow]{r}{\mathrm{sb}} & B_1
  \end{tikzcd} 
\end{equation}
The maps drawn with horizontal arrows are injections, and the ones drawn with vertical arrows are surjections ($\mathrm{sb}$ is a bijection).  Also,
\begin{align*}
	\mu(G^+) &= \beta^{-1} (\eta^+ (S_1))\\
	\mu(G^-)  &= \beta^{-1} (\eta^- (S_1).  
\end{align*}
A section of $\beta$ is the map $\tau$ of \eqref{E:tau}, while a section of $\gamma$ is the map $\xi((q_1, \ldots, q_l)) = (A,B,C)$ with
\begin{align*}
	A &= \phantom{-} \left[ q_2, \ldots, q_l \right]\\
	B &= \phantom{-} \left[ q_1, q_2, \ldots, q_l \right] - \left[ q_2, \ldots, q_{l-1} \right]\\
	C &= - \left[ q_1, \ldots, q_{l-1} \right]
\end{align*}
The image of $\xi$ is the set of forms in $G^+$ with discriminants of the form $k^2 \pm 4$, and the form $(A,B,C)$ given by the above formulas has discriminant 
\begin{equation*}
	\left( \left[ q_1, \ldots, q_l \right] + \left[ q_2,\ldots, q_{l-1} \right] \right)^2 + (-1)^{l+1} \cdot 4.
\end{equation*}
\end{theorem}

\begin{proof}
 If $f$ has discriminant $\Delta$, then $\eta^+ \circ \gamma(f) = (1, q_1, \ldots, q_l)$, where $q_1, \ldots, q_l$ is the sequence of quotients in the regular continued fraction of $z'/Au$, where $t^2 - \Delta u^2 = (-1)^{l} 4$ is the fundamental solution, and $z' = \frac{t+Bu}{2}$.   To compute $\beta \circ \mu (f)$, since $\mu (f)$ also has discriminant $\Delta$, we set $z = \frac{t + (B+2A)u}{2}$, then expand $\frac{z}{z-Au}$ in a continued fraction with quotient sequence of length parity $l+1$. As observed in the proof of Lemma \ref{L:beading}, the quotient sequence of the expansion of $\frac{z}{z-Au}$ is obtained by pinching both ends of the quotient sequence of the expansion of $\frac{z}{Au}$.  Comparing $z'$ and $z$, we see that $\frac{z}{Au}$ expands as a continued fraction with quotient sequence $(q_1 + 1, q_2, \ldots, q_l)$.  Thus, $\beta \circ \mu(f)$ is found by pinching both ends of $(q_1 + 1, q_2, \ldots, q_l)$, or it would be except that the length parity is wrong.  Changing the length parity is accomplished by switching between the two continued fraction expansions of $\frac{z}{z-Au}$. Thus, $\beta \circ \mu(f) = (1, q_1, \ldots, q_l) = \eta \circ \gamma(f)$ and the diagram commutes.

For the second diagram, if $f \in G^-$, then $f^{\mathrm{R}} \in G^+$ and \eqref{E:mu} shows $\mu(f^{\mathrm{R}}) = \mu(f)^{\mathrm{R}}$.  Thus, from Proposition \ref{P:reversal}
\begin{equation*}
	(\beta \circ \mu(f))^{\mathrm{R}}  = \beta\left( \mu(f)^{\mathrm{R}} \right) =   \beta \circ \mu \left(f^{\mathrm{R}} \right)   =   \eta^+ \circ \gamma(f^{\mathrm{R}}) = \left( \eta^-  \circ \gamma \circ \rho(f) \right)^{\R},
\end{equation*}
so $\beta \circ \mu(f) = \eta^{-} \circ \gamma \circ \rho (f)$.

It is easily verified that $\eta^+$, and $\eta^-$ are injections and that restricting $\mu$ to $G^+$ or $G^-$ gives an injection. Theorem \ref{T:formfrombeads} states that $\beta$ is a surjection with section $\tau$.

Let us now verify that that  $\mu(G^+) = \beta^{-1} (\eta^+ (S_1))$. The inclusion $\mu(G^+) \subset \beta^{-1} (\eta^+ (S_1))$ follows from the commutativity of the first diagram. Suppose $\mathfrak{s} = q_1, \ldots, q_l$ is a nonempty natural string and $f = (A,B,C)$ is a Z-reduced form such that
\begin{equation}\label{E:prepended1}
	\beta(f) = \eta^+ (\mathfrak{s}) = 1, q_1, \ldots, q_l.
\end{equation}
Then $f$ is in $\mu(G^+)$ if and only if $f(x-y,y) = (A,B-2A,A-B+C)$ is G-reduced.  Since $f$ is Z-reduced, we have $A (A-B+C) < 0$.  We must check also that $B - 2A > |2A-B+C|$.  Since $C > 0$, it suffices to show $4A-2B+C < 0$.  

Since $f$ is Z-reduced, we have $C - B < 0$, so $4A - 2B + C < 0$ is immediate if $4A - B < 0$.  We may thus assume $4A - B \geq 0$.  If $f$ has discriminant $\Delta$ and $(t,u)$ is the fundamental solution of $|t^2 - \Delta u^2| = 4$, then $\beta(f)$ is computed by setting $z = \frac{t+Bu}{2}$ and expanding $\tfrac{z}{z-Au}$ in a continued fraction.  From \eqref{E:prepended1}, the first quotient is $1$, so either $\beta(f) = (1,1)$ or $Au < \frac{z}{2}$.  If $\beta(f) = (1,1)$, then $f = (1,3,1)$ and $4A - 2B + C = -1 < 0$.  Otherwise, 
\begin{equation*}
	 4A - B < \frac{t}{u},
\end{equation*}
and we have assumed the left side is nonnegative.  Squaring both sides and using the definition of $(t,u)$, we find 
\begin{equation*}
	4A - 2B + C < \pm \frac{1}{Au^2} \leq 1.
\end{equation*}
But $4A-2B+C = f(2,-1)$, and $f$ cannot take the value $0$ since its discriminant is nonsquare.  The desired inequaltiy $4A-2B+C < 0$ follows, and we have shown $\beta^{-1} (\eta^+ (S_1)) \subset \mu(G^+)$.

If instead $\beta (f) = \eta^- \mathfrak{s} = q_1 \ldots q_l 1$, then Proposition \ref{P:reversal} shows $\beta \left( f^{\R} \right) = \eta^+ \left( \mathfrak{s}^{\R} \right)$.  Since $\beta^{-1} (\eta^+ (S_1)) = \mu(G^+)$, we have $f^{\R} = \mu(g)$ for some $g \in G^+$, so $f = (\mu (g))^{\R} = \mu(g^R)$, hence $f \in \mu(G^{-})$.  Thus, $\beta^{-1} (\eta^- (S_1)) = \mu(G^-)$.

We now check that $\xi$ is a section of $\gamma$.  If $\mathfrak{s} = q_1\ldots q_l$ is in $S_1$, then $\tau \circ \eta^+(\mathfrak{s}) = (A',B',C')$ with 
\begin{align*}
	A' &= \left[ q_2, \ldots, q_l \right],\\
	B' &= \left[ q_1 + 1, q_2, \ldots, q_l \right] + \left[ q_2, \ldots, q_{l-1}, q_l - 1 \right], \\
	C' &= \left[ q_1 + 1, q_2, \ldots, q_{l-1}, q_l - 1 \right].  
\end{align*}
This is in $\beta^{-1} (\eta^+ (S_1))$, so equals $\mu(g)$ for $g = (A', B'-2A', A'-B'+C') \in G^+$.
A few applications of \eqref{E:contends} shows that $g$ is the form $\xi(\mathfrak{s})$ in the theorem. Diagram chasing then proves $\gamma \circ \xi (\mathfrak{s}) = \mathfrak{s}$.   This also shows that $\gamma$ is a surjection.

 Theorem \ref{T:formfrombeads} shows that $\tau \circ \eta^+(\mathfrak{s})$ has discriminant 
 \begin{equation*}
 	(\left[ q_1 + 1, q_2, \ldots, q_l \right] - \left[ q_2, \ldots, q_l - 1 \right])^2 + (-1)^{l+1} \cdot 4.
\end{equation*}
But $g = \xi(\mathfrak{s})$ has the same discriminant, and the formula for the discriminant of $\xi(\mathfrak{s})$ in the theorem appears with a couple of applications of \eqref{E:contends}.  This shows that $\xi(S_1)$ is contained in the set of forms in $G^+$ with discriminant of the form $k^2 \pm 4$.  Conversely, if $f$ is such a form, we have just shown  $\xi \circ \gamma(f)$ is the unique form in $g \in G^+$ with $\mu(g) = \tau \circ \eta^+ \circ \gamma(f)$. But
 \begin{equation*}
 	\tau \circ \eta^+ \circ \gamma(f) =  \tau \circ \beta  \circ \mu(f) = \mu(f),
\end{equation*}
so $f = \xi \circ \gamma(f)$.
\end{proof}

\begin{lemma}\label{L:firstcoefficient}
If $f = (A,B,C) \in G^+$ and $\gamma(f) = (q_1, \ldots, q_l)$, then 
\begin{equation*}
	R_G(f) = f(q_1 x + y,-x).
\end{equation*}
\end{lemma}

\begin{proof}
Suppose $f$ has discriminant $\Delta$ and let $(t,u)$ be the fundamental solution of $|t^2 - \Delta u^2| = 4$.  Since $R_G$ commutes with scalar multiplication, Lemma \ref{L:reducetouf} shows that it suffices to prove the lemma when $u=1$ and $f$ has discriminant $t^2 \pm 4$.  

From Theorem \ref{T:xi}, we have $f =  \xi ((q_1, \ldots, q_l))$,  $t = \left[ q_1, \ldots, q_l \right] + \left[ q_2, \ldots, q_{l-1} \right]$, and
\begin{equation}\label{E:somevalues}
	A = \left[ q_2, \ldots, q_l \right], \qquad 
	\frac{t+B}{2} = \left[ q_1, \ldots, q_l \right], \qquad \frac{t-B}{2} = \left[ q_2, \ldots, q_{l-1} \right].
\end{equation}

We verify that
\begin{equation}\label{E:doubleinequality}
	(2q_1 A - B)^2 < t^2 + (-1)^{l+1} \cdot 4 < (2(q_1 + 1)A - B)^2.
\end{equation}
When $l=1$, this reduces to $q_1^2 < q_1^2 + 4 < q_1^2 + 4q_1 + 4$, so let us assume that $l \geq 2$.  We readily check that $q_1 A - B = \left[ q_2, \ldots, q_{l-1} \right] - \left[ q_3, \ldots, q_l \right]$, and then
\begin{equation*}
	q_1 (q_1 A - B) < [q_1, ..., q_{l-1}] < (q_1+1) (A(q_1+1) - B).
\end{equation*}
Multiplying through by $4A = 4\left[q_2, \ldots, q_l \right]$ and using \eqref{E:contdet}, we obtain
\begin{equation*}
	4q_1^2 A^2 - 4q_1 AB  < t^2 - B^2 + (-1)^{l+1} \cdot 4 < 4(q_1+1)^2 A^2 - 4(q_1+1)AB
\end{equation*}
and \eqref{E:doubleinequality} follows.  The discriminant of $f$ is $\Delta = t^2 + (-1)^{l+1} \cdot 4$, so since $2q_1 A - B > 0$, we have
\begin{equation*}
	q_1 < \frac{B + \sqrt{\Delta}}{2A} < q_1 + 1,
\end{equation*}
and the lemma follows.
\end{proof}

\begin{theorem}\label{T:reductionrelation}
Suppose $f \in G^+$ with $\gamma(f) = (q_1, \ldots, q_l)$ and $g \in Z$.  Then
\begin{alignat*}{3}
	\gamma \left( \rho \circ R_G (f) \right) &=  T_G (\gamma(f)), \qquad &
	\gamma ( R_G^2 (f)) &= T_G^2 (\gamma(f)), \qquad &
	\beta (R_Z (g)) &= T_Z (\beta(g)),\\
	\mu (R_G (f)) &= R_Z (\mu(f)), \qquad &
	\mu (R_G^2 (f)) &= R_Z^{q_2} (\mu(f)). &
\end{alignat*}   
(An exponent on an operator indicates iteration.)
\end{theorem}

\begin{proof}
Let $f$ have discriminant $\Delta$ and let $(t,u)$ be the fundamental solution of $|t^2 - \Delta u^2| = 4$.  Noting that $R_G$ and $\rho$ commute with scalar multiplication, we find using Lemma \ref{L:reducetouf} that $\gamma \left( \rho \circ R_G (uf) \right) = \gamma \left(u \left( \rho \circ R_G (f) \right) \right) = \gamma \left( \rho \circ R_G (f) \right)$ and $T_G (\gamma(uf)) = T_G (\gamma(f))$.   Thus, in proving the first relation, we may assume without loss of generality that $f$ has discriminant $t^2 \pm 4$ and $u = 1$.  

Suppose that $f = (A,B,C)$ and $\gamma(f) = (q_1, \ldots, q_l)$.  Theorem \ref{T:xi}  shows that $f =  \xi ((q_1, \ldots, q_l))$,  $t = \left[ q_1, \ldots, q_l \right] + \left[ q_2, \ldots, q_{l-1} \right]$, and \eqref{E:somevalues} again holds.

From Lemma \ref{L:firstcoefficient}, we have $\rho \circ R_G (f) = (-q_1^2 A + q_1 B - C, 2q_1 A - B, -A)$.  This form has discriminant $t^2 \pm 4$, so we compute $\gamma (\rho \circ R_G(f))$ by setting 
\begin{equation*}
	z = \frac{t-B + 2q_1A}{2} = \left[ q_2, \ldots, q_{l-1} \right] + q_1 \left[ q_2, \ldots, q_l \right] = \left[ q_2, \ldots, q_{l}, q_1 \right]
\end{equation*}
and expanding $z/(-q_1^2 A + q_1 B - C)$ in a continued fraction.  This denominator equals
\begin{align*}
	&-q_1^2 \left[ q_2, \ldots, q_l \right] + q_1 \left[ q_1, \ldots, q_l \right] - q_1 \left[ q_2, \ldots, q_{l-1} \right] + \left[ q_1, \ldots, q_{l-1} \right] \\
	&= q_1 \left[ q_3, \ldots, q_l \right] + \left[ q_3, \ldots, q_{l-1} \right] = \left[ q_3, \ldots, q_l, q_1 \right].
\end{align*}
Thus, 
\begin{equation*}
	\gamma \left( \rho \circ R_G (f) \right) = (q_2, \ldots, q_l, q_1) = T_G (\gamma(f)).
\end{equation*}

The relation $\gamma \left( R_G^2 (f) \right) = T_G^2 (\gamma(f))$ follows immediately from the first, noting that $\rho$ and $R_G$ commute.  The relation $\beta (R_Z(f)) = T_Z (\beta (f))$ is the content of Lemma \ref{L:beading}.  

To check the final two relations, we can again reduce to when $u = 1$ and $f$ has discriminant $t^2 \pm 4$. Theorem \ref{T:xi} gives
\begin{equation*}
	\beta \circ \mu \left( R_G (f) \right) = \eta^- \circ \gamma \circ \rho \left( R_G (f) \right), \qquad \beta \circ \mu \left( R_G^2 (f) \right) = \eta^+ \circ \gamma \left( R_G^2 (f) \right).
\end{equation*}
The first two relations then give
\begin{align*}
	\beta \circ \mu \left( R_G (f) \right)  &= \eta^- \circ T_G (\gamma(f)) = (q_2, \ldots, q_l, q_1, 1),\\
	\beta \circ \mu \left( R_G^2 (f) \right) &= \eta^+ \circ T_G^2 (\gamma(f)) = (1, q_3, \ldots, q_l, q_1, q_2).
\end{align*}

On the other hand, the third relation and Lemma \ref{L:beading} give
\begin{align*}
	\beta \left( R_Z (\mu(f)) \right) &= T_Z (\beta \circ \mu(f)) = T_Z (\eta^+ \circ \gamma(f)) = (q_2, \ldots, q_l, q_1, 1), \\
	\beta (R_Z^{q_2} (\mu(f))) &= T_Z^{q_2} (\beta \circ \mu(f)) = T_Z^{q_2} (\eta^+ \circ \gamma(f)) = T_Z^{q_2} ((1, q_1, \ldots, q_l)) = (1, q_3 \ldots, q_l, q_1, q_2).
\end{align*}
The final two relations then follow from Theorem \ref{T:formfrombeads}.
\end{proof}

\begin{corollary}\label{C:munotinjective} 
If $f$ and $g$ are G-reduced forms such that $\mu(f) = \mu(g)$, then one of $f$ and $g$ is in $G^+$ and the other is in $G^-$.  If $g \in G^-$, then 
\begin{equation*}
	f = R_G (g) = g(-x+y,-x).
\end{equation*}
Conversely, if $g \in G^-$ and $f = g(-x+y,-x)$, then $\mu(f) = \mu(g)$.
\end{corollary}

\begin{proof}
To verify the first statement, we examine \eqref{E:mu} and observe that $\mu$ is injective when restricted to either $G^+$ or $G^-$.  For the rest, we again reduce to the case when $f$ and $g$ have discriminants of the form $t^2 \pm 4$.  So suppose $g \in G^-$, $f \in G^+$ and $\mu(f) = \mu(g)$.  We apply $\beta$ and use Theorem \ref{T:xi} to obtain
\begin{equation*}
	\eta^+ \circ \gamma(f) = \eta^- \circ \gamma \circ \rho (g).
\end{equation*}
There exists then a natural string $\mathfrak{s}$ such that $\gamma \circ \rho (g) = 1 \mathfrak{s}$ and $\gamma(f) = \mathfrak{s} 1$, so 
\begin{equation*}
	 \gamma(f) = T_G \circ \gamma \circ \rho (g) = \gamma \circ \rho \circ R_G \circ \rho (g) = \gamma (R_G (g)).
\end{equation*}
Theorem \ref{T:xi} shows that $\gamma$ is a bijection when restricted to forms with discriminants of the form $t^2 \pm 4$, so $f = R_G (g)$.  Let $g' = \rho (g) \in G^+$.  Since $\gamma  (g') = 1 \mathfrak{s}$, Lemma \ref{L:firstcoefficient} shows that $\rho \circ R_G (g) = R_G (g') = g'(x+y,-x)$, and $R_G (g) = g(-x + y, -x)$ follows.  

The converse is checked with a quick computation.
\end{proof}

\begin{remarks}
1.  The corollary can be proved directly.  If $f=(A,B,C)$, $A> 0$ and $g=(A',B',C')$, $C' > 0$, then $\mu(f) = \mu(g)$ leads to the relations $A = A' + B' + C'$, $B+2A = B'+2C'$, and $A+B+C=C'$.  These readily imply $C=A'=-\frac{B+B'}{2}$, and from this, $f = g(-x+y,-x)$.  It can then be shown that $f = R_G(g)$.\\

2. The corollary is visually evident from the Conway topograph of $\mathfrak{c}$ once we learn how to see G- and Z-reduced forms:  G-reduced forms correspond to ``riverbends'' \cite{mW2017} and Z-reduced forms correspond to ``positive confluences'', i.e., inlets to the river from the side with positive numbers.  Every riverbend is adjacent to exactly one positive confluence, and $\mu$ maps the corresponding G-reduced form to the corresponding Z-reduced form.  Pairs of forms that have the same image under $\mu$ correspond to adjacent riverbends, and being adjacent forces the conditions of the corollary.
\end{remarks}

\section{Constructing minimal periods with $\sigma$}\label{S:Denjoy}

In this final section, we prove that $\sigma$ produces minimal periods of Denjoy continued fractions.

\begin{definition}
A \textbf{Denjoy continued fraction} is an expression
\begin{equation*}
	q_1 + \cfrac{1}{q_2 + \cfrac{1}{\ddots}}
\end{equation*}
in which all quotients are $0$ or $1$ and such that the quotient sequence does not include two consecutive zeroes.   If $\xi > 0$ is an irrational number, then $\xi$ expands uniquely in a Denjoy continued fraction as follows.  Set $\xi_0 = \xi$.  For $n \geq 1$, we define $\xi_n$ and $q_n$ recursively by setting 
\begin{equation*}
	q_n = 0 \quad \text{ if \,  $\xi_{n-1} < 1$}, \qquad q_n = 1\quad \text{ if \,  $\xi_{n-1} > 1$}
\end{equation*}
and $\xi_n = \frac{1}{\xi_{n-1} - q_n}$.
\end{definition}

The replacement rules
\begin{equation*}
	\ddots + \cfrac{1}{q_{n-2} + \cfrac{1}{1 + \cfrac{1}{0}}} = \ddots + \cfrac{1}{q_{n-2}}, \quad \text{ and } \quad a + \cfrac{1}{0 + \cfrac{1}{b}} =  a+b
\end{equation*}
let us deal with convergents whose final quotient is $0$ or convert between Denjoy and regular continued fractions. If $\xi > 0$, then to convert the regular continued fraction of $\xi$ into a Denjoy one, a process we call \emph{regular-to-Denjoy conversion}, we simply replace each natural partial quotient $q_i$ by the string $1010\ldots01$ containing $q_i-1$ zeroes (and we leave an initial quotient of 0 alone).

From the criterion for periodicity of a regular continued fraction, we obtain:
\begin{theorem}\label{T:conversion}
The sequence of quotients of the Denjoy continued fraction of an irrational number $\xi > 0$ is eventually periodic if and only if $\xi$ is a quadratic irrational.  If the regular continued fraction is purely periodic, then so is the Denjoy one.  Indeed, applying regular-to-Denjoy conversion to the minimal period of the regular continued fraction gives the minimal period of the Denjoy one.
\end{theorem}

We will also need to convert negative infinite continued fractions into regular ones (\emph{negative-to-regular conversion}).    Let us write $(q_1, q_2, \ldots)$ and $[q_1, q_2, \ldots]$ to distinguish sequences of quotients of regular and negative continued fractions.  Thus, an irrational number $\xi$ corresponds to $[q_1, q_2, \ldots]$ when 
\begin{equation*}
	\xi = q_1 - \cfrac{1}{q_2 - \cfrac{1}{\ddots}}.
\end{equation*}
The simple algorithm for converting between negative continued fractions and regular ones is classical \cite{mS1864}.  If $[q_1, q_2, \ldots]$ corresponds to an irrational number $\xi > 1$, then $q_i \geq 2$ for all $i$.  The corresponding regular continued fraction begins $(q_1-1, \ldots)$.  To determine what follows the second term, we count the maximal string of $2$'s following $q_1$ in $[q_1, q_2, \ldots]$ -- say there are $n_1$ of these $2$'s. The term $q_{n_1+2}$ following this string is $\geq 3$, and we let $m_1 = q_{n_1+2} - 2$.  We then count the maximal string of $2$'s following $q_{n_1+2}$ -- say there are $n_2$ of them. We then let $m_2$ be $2$ less than the term $q_{n_1+n_2+3}$ following this string of $2$'s.  Continuing, we obtain sequences $(n_i)$ and $(m_i)$, and $\xi$'s regular continued fraction has quotient sequence $(q_1-1, n_1, m_1, n_2, m_2, \ldots)$.  

\begin{theorem*}[Lagrange, Galois, Zagier \cite{dZ1981}]
If $g = (A,B,C)$ has discriminant $\Delta$ and $\xi = \tfrac{B+\sqrt{\Delta}}{2A}$, then the following are equivalent:
\begin{itemize}
	\item $g \in G^+$
	\item $\xi$ has purely periodic regular continued fraction expansion,
	\item $\xi > 1$ and $-1 < \overline{\xi} < 0$ ($\overline{\xi}$ being the conjugate of $\xi$.
\end{itemize}
Similarly, the following are equivalent
\begin{itemize}
	\item $g \in Z$
	\item $\xi$ has purely periodic negative continued fraction expansion,
	\item $\xi > 1$ and $0 < \overline{\xi} < 1$ .
\end{itemize}
In the first case, we say $\xi$ is \emph{G-reduced}, and in the second we say it is \emph{Z-reduced}.
\end{theorem*}

\begin{remark}
If we say $\xi > 0$ is \emph{D-reduced} if it has purely periodic Denjoy continued fraction expansion. Theorem \ref{T:Denjoy} below shows that if $\xi$ is Z-reduced, then $\xi - 1$ is \emph{D-reduced}.  The above result shows the set of G-reduced irrationalities is strictly contained in the set of these $\xi-1$.  It is not hard to show that that the latter set is strictly contained in the set of D-reduced irrationalities, with the remaining D-reduced numbers being those $\xi$ satisfying $\xi > 1$ and $\overline{\xi} < -1$.
\end{remark}

\begin{theorem}\label{T:Denjoy}
If $(A,B,C) \in Z$ has discriminant $\Delta$, then the quadratic irrationality
\begin{equation*}
	\xi = \frac{B-2A+\sqrt{\Delta}}{2A}
\end{equation*}
has purely periodic Denjoy continued fraction, and the minimal period is obtained from $\sigma((A,B,C))$ by replacing each `0' by `01'.
\end{theorem}

\begin{proof}
  From Zagier's theorem above and the algorithm for negative-to-regular conversion, if $\xi > 1$, then $\tfrac{B+\sqrt{\Delta}}{2A} > 2$ and $\xi$ has a purely periodic regular continued fraction expansion. Theorem \ref{T:conversion} shows $\xi$ has purely periodic Denjoy continued fraction.  

Now since $(A,B,C) \in Z$, then if $\xi > 1$, the theorem of Lagrange-Galois-Zagier shows that $\xi$ is $G$-reduced, and it follows that $(A,B,C) = \mu(g)$, where $\xi$ is the $G$-reduced quadratic irrationality corresponding to $g$. 
 Theorem \ref{T:xi} shows that $\beta((A,B,C)) = (1,\gamma(g))$.  Let $\mathfrak{s}$ be the binary string obtained when we replace each `0' in $\sigma((A,B,C)) = \mathrm{sb}(1,\gamma(g))$ by `01'.  It is then readily checked that applying regular-to-Denjoy conversion to the finite natural string $\gamma(g)$ produces $\mathfrak{s}$.  Property 1 of $\gamma$ and Theorem \ref{T:conversion} then show that the theorem holds in this case.

If $0 < \xi < 1$, then $(A,B,C) \not\in \mu(G^+)$ and $\beta((A,B,C)) = (q_1, \ldots, q_l)$ with $q_1 \geq 2$.  This is equivalent to the reducing number of $(A,B,C)$ being $2$.  With a look at the operation \eqref{E:TZ}, we see that while iterating the reduction operator, we will continue to have a reducing number $2$ through $q_1-1$ iterations, at which point we reach a form in $\mu(G^+)$, say $\mu(h)$.  Thus, $\gamma(h) = (q_2, \ldots, q_{l-1}, q_1+q_l-1)$ and $(A,B,C)$ is carried to $h$ by operating by the matrix
\begin{equation*}
	M = \begin{bmatrix} 2 & 1\\ -1 & 0 \end{bmatrix}^{q_1-1} \begin{bmatrix} 1 & -1\\ 0 & 1 \end{bmatrix} = \begin{bmatrix} q_1 & q_1 - 1 \\ -(q_1 - 1) & -(q_1 - 2) \end{bmatrix}.
\end{equation*}

Now some algebra will show that $\frac{1}{1/\xi - (q_1-1)}$ is the quadratic irrational corresponding to the form obtained from $(A,B,C)$ by applying $M$, i.e., $h$.  It follows that $\frac{1}{1/\xi - (q-1)}$ is $G$-reduced and has purely periodic regular continued fraction expansion with minimal period $\gamma(h) = (q_2, \ldots, q_{l-1}, q_1+q_l-1)$.  Thus, $\xi$ has regular continued fraction $(0, q_1 - 1, \overline{q_2, \ldots, q_{l-1}, q_1 + q_l - 1})$.  Converting using regular-to-Denjoy conversion, we see that $\xi$ has purely periodic Denjoy continued fraction expansion with minimal period 
\begin{equation*}
	(0,1)^{q_1-1}, 1, (0,1)^{q_2 - 1}, 1, (0,1)^{q_3-1}, \ldots, (0,1)^{q_l - 1},
\end{equation*}
(in which $(0,1)^{a}$ is shorthand for the $a$-fold self-concatenation of the string 01).  This is also the sequence obtained from $\sigma((A,B,C)) = \mathrm{sb}((q_1, \ldots, q_l)$ upon replacing each `0' by `01', and the proof is complete.
\end{proof}

 \end{document}